\documentclass[11pt]{amsart}
\usepackage{amscd,setspace,amssymb,amsopn,amsmath,amsthm,mathrsfs,lmodern,graphics,amsfonts,enumerate,verbatim,calc
}
\usepackage[all]{xy}
\usepackage[colorlinks=true, citecolor=PineGreen, filecolor=PineGreen, linkcolor=Purple,pagebackref, hyperindex]{hyperref}
\usepackage[dvipsnames]{xcolor}
\usepackage{todonotes}
\usepackage{tikz-cd}
\usepackage{color, hyperref}
\usepackage[OT2,OT1]{fontenc}

\newcommand\cyr{%
\renewcommand\rmdefault{wncyr}%
\renewcommand\sfdefault{wncyss}%
\renewcommand\encodingdefault{OT2}%
\normalfont
\selectfont}
\DeclareTextFontCommand{\textcyr}{\cyr}
\usepackage{amssymb,amsmath}
\DeclareFontFamily{OT1}{rsfs}{}	
\DeclareFontShape{OT1}{rsfs}{n}{it}{<-> rsfs10}{}
\DeclareMathAlphabet{\fmathscr}{OT1}{rsfs}{n}{it}
\numberwithin{equation}{section}
\newtheorem*{MainTheorem}{Main Theorem}

\newtheorem{theorem}{Theorem}[section]
\newtheorem{lemma}[theorem]{Lemma}

\newtheorem{corollary}[theorem]{Corollary}

\newtheorem{question}{Question}

\theoremstyle{definition}
\newtheorem{definition}[theorem]{Definition}
\newtheorem{remark}[theorem]{Remark}
\theoremstyle{remark}

\newcommand{\im}{\operatorname{Im}}

\newcommand{\Spec}{\operatorname{Spec}}

\newcommand{\Height}{\operatorname{ht}}

\newcommand{\Hom}{\operatorname{Hom}}

\newcommand{\Ann}{\operatorname{Ann}}

\newcommand{\Tr}{\operatorname{Tr}}

\newcommand{\bQ}{\mathbb{Q}}
\newcommand{\bF}{\mathbb{F}}

\newcommand{\fm}{\mathfrak{m}}
\newcommand{\fp}{\mathfrak{p}}
\newcommand{\fq}{\mathfrak{q}}

\newcommand{\fn}{\mathfrak{n}}

\begin{document}
\title{Compatible ideals in $\mathbb{Q}$-Gorenstein rings}

\author[Thomas Polstra]{Thomas Polstra}
\thanks{Polstra was supported in part by NSF Postdoctoral Research Fellowship DMS $\#1703856$, NSF Grant DMS \#101890, and a grant from the Simons Foundation, Grant Number 814268, MSRI}
\address{Department of Mathematics, University of Alabama, Tuscaloosa, AL}
\email{tmpolstra@ua.edu}

\author[Karl Schwede]{Karl Schwede}
\thanks{Schwede was supported in part by NSF CAREER Grant DMS \#1252860/1501102, NSF Grants \#1801849, \#1952522, \#2101800 and a Simons Fellowship.}
\address{Department of Mathematics, University of Utah, Salt Lake City, UT 84102 USA}
\email{schwede@math.utah.edu}


\begin{abstract}
Suppose $R$ is a $F$-finite and $F$-pure $\bQ$-Gorenstein local ring of prime characteristic $p>0$. We show that an ideal $I\subseteq R$ is uniformly compatible ideal (with all $p^{-e}$-linear maps) if and only if exists a module finite ring map $R\to S$ such that the ideal $I$ is the sum of images of all $R$-linear maps $S\to R$.  In other words, the set of uniformly compatible ideals is exactly the set of trace ideals of finite ring maps.
\end{abstract}

\maketitle

\section{Introduction}

Compatibly Frobenius split ideals and subvarieties have played an important role in the study of rings and varieties in characteristic $p > 0$.  They first formally appeared in \cite{MehtaRamanathanFrobeniusSplittingAndCohomologyVanishing} in their study of Schubert varieties, although they also implicitly played a central role in \cite{FedderFPureRat}, at the dawn of the theory of characteristic $p > 0$ singularities.  Within that theory, some very important ideals $I$ are always ``uniformly compatible'' in the sense that for every $\phi : F^e_* R \to R$, we have   
\[
  \phi(F^e_* I) \subseteq I.
\]
The \emph{test ideal}\footnote{technically, the non-finitistic test ideal aka big test ideal} \cite{VassilevTestIdeals,SchwedeCentersOfFPurity} is the smallest nonzero\footnote{at any minimal prime} compatible ideal while the \emph{splitting prime} is the largest compatible proper ideal, \cite{AberbachEnescuStructureOfFPure}.  Being compatibly split is also a central part of the theory of Frobenius split varieties \cite{MehtaRamanathanFrobeniusSplittingAndCohomologyVanishing,BrionKumarFrobeniusSplitting}.  On the other hand, it turns out that the compatibly split ideals are also closely related to the theory of \emph{log canonical centers} from birational complex geometry, \cite{SchwedeCentersOfFPurity}.  Thus, as we begin to move into the world of mixed characteristic singularities, it behooves us to look for other characterizations of these important special ideals.

One other characterization of the test ideal, at least in a $\mathbb{Q}$-Gorenstein domain, is that it is the smallest possible nonzero image 
\[
\Hom_R(S, R) \xrightarrow{eval@1} R
\]
where $S \supseteq R$ is a finite extension, \cite{BlickleSchwedeTuckerTestAlterations, SmithTightClosureParameter}.  But what about the other compatible ideals? Are they also such images?    We answer this question affirmatively in the case $R$ is local and $\bQ$-Gorenstein.  It is well known that any such image is uniformly compatible.

To do this, we prove a more finely tuned version of the celebrated \emph{Equational Lemma}, killing certain cohomology classes while leaving others nonzero (by keeping the extension \'etale over certain primes).
Let $R$ be a Noetherian $F$-finite ring of prime characteristic $p>0$. Hochster's and Huneke's Equational Lemma, \cite[Theorem~2.2]{HochsterHunekeInfiniteIntegralExtensionsAndBigCM} allows one to trivialize relations on parameters of $R$ inside a finite extension of $R$.  Consequently, the absolute integral closure of $R$ is a big Cohen-Macaulay algebra (see also \cite{BhattAbsoluteIntegralClosure}).  
In fact, Hueneke and Lyubeznik \cite{HunekeLyubeznikAbsoluteIntegralClosure} showed that one can even kill all lower local cohomology in a single finite extension  instead of going all the way to $R^+$, see \cite{SannaiSinghGaloisExtensions, BhattAnnihilatingCohomology, BhattDerivedDirectSummand} for generalizations.

Instead of killing intermediate local cohomology however, we are interested in studying the top local cohomology and killing cohomology classes that belong to the tight closure of zero.  This translates to constructing the parameter test module, see \cite{SmithTightClosureParameter}, and via finite covers, corresponds to the test ideal, as done in \cite{BlickleSchwedeTuckerTestAlterations}.  


%





\begin{MainTheorem}[Corollary \ref{cor.GeneralCase}]
\label{Main Theorem}
Let $R$ be a Noetherian local $F$-finite and $F$-pure $\bQ$-Gorenstein ring of prime characteristic $p>0$. Suppose $I\subseteq R$ is a uniformly compatible ideal of $R$. Then there exists a finite ring map $R\to S$ so that $I=\im(\Hom_R(S,R)\to R)$.  Furthermore, if $I$ is not contained in a minimal prime then we may take $R \to S$ to be an extension.  
\end{MainTheorem}


It would be natural to expect this to hold without the $\bQ$-Gorenstein assumption.  That appears to be out of reach with our current techniques.  Indeed, if the Main Theorem were known to hold without the $\bQ$-Gorenstein hypothesis, then every splinter would have the property that the big test ideal equals $R$, that is $\tau(R) = R$ (since the only trace ideal of a splinter $R$ is $R$ itself).  Hence every splinter would be strongly $F$-regular and in particular weakly $F$-regular rings would also be strongly $F$-regular.  For some results related to these conjectures see for example \cite{WilliamsUniformStabilityOfKernels,LyubeznikSmithStrongWeakFregularityEquivalentforGraded,SinghQGorensteinSplinters,AberbachMacCrimmonSomeResultsOnTestElements,LyubeznikSmithCommutationOfTestIdealWithLocalization,ChiecchioEnescuMillerSchwedeTestIdealsFinitelyGenerated,AberbachPolstraLocalCohomology}.

\subsection*{Acknowledgements}  The authors would like to thank Bhargav Bhatt and Javier Carvajal-Rojas for valuable discussions.  We also thank Javier Carvajal-Rojas and Neil Epstein for valuable comments on a previous draft. We thank an anonymous referee for bringing to our attention a mistake in \autoref{Prime ideals in a quasi-Gorenstein ring} in a previous draft of the article.

\section{Background}

Suppose $R$ is an $F$-finite ring of prime characteristic $p>0$.  The functor $F^e_*(-)$ is the restriction of scalars functor along the $e$-iterated Frobenius map $R \xrightarrow{F^e} R$.  Note if $R$ is reduced then as an $R$-module $F^e_* R$ can be identified with $R^{1/p^e}$ viewed as an $R$-module ($R^{1/p^e}$ is the ring of all $p^e$th roots of elements of $R$) and we will switch between these notations as we find convenient.

An ideal of $I\subseteq R$ is said to be \emph{(uniformly) compatible} if for all $e\in\mathbb{N}$ every $\varphi\in\Hom_R(F^e_*R,R)$ 
satisfieds $\varphi(F^e_*I)\subseteq I$. Note in this case $\varphi$ induces a map $F^e_* R/I \to R/I$.  Every uniformly compatible ideal in an $F$-pure ring is radical. We sketch the argument now.  If $x^n\in I$ then $x^{p^e}\in I$ for all $e\gg 0$. 
Because $R$ is $F$-pure for every $e\in\mathbb{N}$ there exits $\varphi\in\Hom_R(F^e_*R,R)$ such that $\varphi(F^e_*1)=1$. Therefore, since we are assuming the ideal $I$ is uniformly compatible we have that $\varphi(F^e_*x^{p^e})=x\varphi(F^e_*1)=x\in I$. 
Moreover, uniformly compatible ideals are easily seen to be closed under finite sums and finite intersections. In fact, it follows from this that there are finitely many uniformly compatible ideals in an $F$-pure ring by 
\cite[Theorem~3.1]{EnescuHochsterTheFrobeniusStructureOfLocalCohomology} and \cite{SharpGradedAnnihilatorsOfModulesOverTheFrobeniusSkewPolynomialRing} in the local case.  
In fact, there are even bounds on how many such ideals there can be when $R$ is local \cite{SchwedeTuckerNumberOfFSplit,HuenekeWatanabeUpperBoundsOnMultiplicity}.
There are also finitely my uniformly compatible ideals without assuming $R$ is local by \cite{SchwedeFAdjunction,KumarMehtaFiniteness}.

\begin{definition}[The \emph{trace ideal}]
    Suppose $R$ is a ring and $S$ is an $R$-algebra.  The \emph{trace ideal} of $S$, denoted $\tau_{S/R}$ is the image of the evaluation-at-1 map:
    \[
        \Hom_R(S, R) \to R.
    \]
    This ideal is also called the \emph{order ideal} in \cite{EvansGriffith}.
    Note, in the case that $S$ is a finite $R$-module, then $\Hom_R(S, R) = \omega_{S/R}$ the relative canonical module.
\end{definition}




The following lemma is well known to experts, for instance it is implicit throughout \cite{BlickleStablerFunctorialTestModules} and \cite{CaravajalRojasStablerOnBehaviorUnderFiniteCovers}.  The lemma states that the trace ideal of a finite extension of $R$ defines a uniformly compatible ideal of $R$.  We include a proof for the convenience of the reader.

\begin{lemma}
\label{Lemma image ideals are compatible}
Let $R$ be an $F$-finite ring of prime characteristic $p>0$. Suppose $R\to S$ is a finite extension of $R$ and $\tau_{S/R}$ is the trace ideal. Then $\tau_{S/R}$ is a uniformly compatible ideal of $R$.
\end{lemma}
\begin{proof}
    Fix some $\phi \in \Hom_R(F^e_* R, R)$.  The result follows immediately from the commutativity of the following diagram:
    \[
        \xymatrix{
            \Hom_{F^e_* R}(F^e_* S, F^e_* R) \ar[dr]  \ar@{=}[r] & F^e_* \Hom_R(S, R)  \ar[d] \ar[r]^{\widetilde \phi} & \Hom_R(S, R) \ar[d] \\
            & F^e_* R \ar[r]_{\phi} & R
        }
    \]
    Here the vertical maps are obtained by evaluation at 1, and $\widetilde{\phi}$ is obtained as the following composition:
    \[
        \widetilde{\phi} : \Hom_{F^e_* R}(F^e_* S, F^e_* R) \hookrightarrow \Hom_{R}(F^e_* S, F^e_* R) \xrightarrow{\substack{\text{restrict} \\ \text{source}}} \Hom_R(S, F^e_* R) \xrightarrow{\Hom_R(S, \phi)} \Hom_R(S, R).
    \]
\end{proof}

Since the trace ideal is uniformly compatible, in an $F$-pure ring, all trace ideals are radical and $R$ modulo a trace ideal is automatically $F$-pure as well.

We recall the notion of a quasi-Gorenstein and $\bQ$-Gorenstein ring, as well as Weil divisorial modules.

\begin{definition}
    An S2 ring is called \emph{quasi-Gorenstein} (or $1$-Gorenstein) if it has a canonical module $\omega_R$ which is locally free.  
    
    A G1\footnote{Gorenstein in codimension 1} and S2 ring is called \emph{$\bQ$-Gorenstein} if some symbolic (equivalently reflexive or S2-ified) power of  $\omega_R$ is locally free.  A \emph{Weil divisorial module} for a G1 and S2 ring is a fractional ideal $N$, nonzero at any minimal prime of $R$, that is S2 as an $R$-module and which is locally free of rank 1 in codimension 1.  These are called \emph{almost Cartier divisors} in \cite{HartshorneGeneralizedDivisorsOnGorensteinSchemes}.  
    The \emph{index} of a Weil divisorial module $N$ is the smallest integer $n$ such that $N^{(n)}$ is locally free (here $(-)^{(n)}$ means symbolic power, or equivalently reflexive power). 
    The \emph{index} of a $\bQ$-Gorenstein ring $R$ is the index of $\omega_R$ viewed as a Weil divisorial module (embedded into $K(R)$).  In other words, it is the smallest $n > 0$ such that $\omega_R^{(n)}$ is locally free.  
\end{definition}

\begin{lemma}
\label{Lemma checking containment}
Let $R$ be an $F$-finite and quasi-Gorenstein ring of prime characteristic $p>0$ and let $Q \in \Spec(R)$. Suppose $R \rightarrow S$ is a finite ring map. Then the following are equivalent:
\begin{enumerate}
\item The trace ideal $\tau_{S/R}$ is contained in $Q$.
\item The map of local cohomology modules $H^{\Height(Q)}_{QR_Q}(R_Q)\to H^{\Height(Q)}_{QR_Q}(S_Q)$ is not injective.
\end{enumerate}
\end{lemma}
\begin{proof}
Without loss of generality we may assume $R=(R,\fm)$ is a complete local ring of Krull dimension $d$ with maximal ideal $Q=\fm$. The Matlis dual of $H^d_\fm(R)\to H^d_\fm(S)$ is the trace map $\Tr_{S/R}: \omega_{S/R} = \Hom_R(S, R) \to R$.  The local cohomology map is injective if and only if the dual map is surjective, in other words if the image is not contained in $\fm$. 
\end{proof}

Suppose $(R,\fm,k)$ is a local quasi-Gorenstein ring of prime characteristic $p>0$ and Krull dimension $d$. Then Matlis duality provides to us a one-to-one correspondence between compatible ideals of $R$ and Frobenius stable submodules of $H^d_\fm(R)$, see \cite[Proposition~5.2]{BlickleBockleCartierModulesFiniteness} 
for details. We record this correspondence as a lemma for future reference.

\begin{lemma}
\label{Lemma Frobenius stable submodules corresponds to compatible ideals}
Let $(R,\fm,k)$ be an $F$-finite and complete local quasi-Gorenstein ring of prime characteristic $p>0$. Then there is a one-to-one correspondence between compatible ideals of $R$ and Frobenius stable submodules of $H^d_\fm(R)$. If $(-)^\vee$ denotes Matlis duality then the correspondence is given by
\begin{enumerate}
\item If $I$ is a uniformly compatible ideal of $R$ then $(R/I)^\vee$ is a Frobenius stable submodule of $H^d_\fm(R)$.
\item If $N\subseteq H^d_\fm(R)$ is Frobenius stable submodule of $H^d_\fm(R)$ then $\Ann_R(N^\vee)$ is a uniformly compatible ideal of $R$.
\end{enumerate}
\end{lemma}


\section{Proof of Main Theorem}

\begin{theorem}
\label{Prime ideals in a quasi-Gorenstein ring}
Let $R$ be an $F$-finite and $F$-pure quasi-Gorenstein ring of prime characteristic $p>0$ and suppose that $I\subseteq R$ is a uniformly compatible ideal. Then there exists a finite map $R\to S$ (an extension if $I$ is not contained in any minimal prime) such that $I=\im(\Hom_R(S,R)\to R)$. If $R$ is a normal domain and $I \neq 0$, then $S$ can be chosen to be a domain.  Even better, for each compatible prime $Q$ with $I \not\subseteq Q$ we have that $R_Q \to S_Q$ is \'etale.
\end{theorem}

Our proof is closely related to the method of \cite{HunekeLyubeznikAbsoluteIntegralClosure}.  However because of our assumptions we are able to use prime avoidance to control the form of the equations that the elements we are adjoining satisfy.  

\begin{proof}
Recall that compatible ideals in an $F$-pure ring are always radical ideals. Suppose $I=P_1\cap \cdots \cap P_t$ and each $P_j\in\Spec(R)$. Let $\mathcal{Q}=\{Q_1,\ldots, Q_m\}$ be the finitely many compatible prime ideals of $R$ which $I$ is not contained in. We will show that there exists a finite map $R\to S$ so that $I=\im(\Hom_R(S,R)\to R)$ by constructing a finite $R$-algebra $S$ so that
\begin{enumerate}
\item $\im(\Hom_R(S,R)\to R)\subseteq I$
\item and $R\to S$ is \'{e}tale at all primes $Q\in \mathcal{Q}$.
\end{enumerate}
This will show $I=\im(\Hom_R(S,R)\to R) = \tau_{S/R}$. Indeed, if $R\to S$ is \'{e}tale at each $Q\in\mathcal{Q}$ then $R\to S$ splits at each $Q\in\mathcal{Q}$ and the trace ideal of $R\to S$ cannot be contained in such primes. By Lemma~\ref{Lemma image ideals are compatible} the ideal $\tau_{S/R}$ is radical and therefore must agree with $I$ since its prime components are all compatible by Lemma~\ref{Lemma image ideals are compatible}.

Consider the local ring $R_{P_j}$. Suppose that $R_{P_j}$ is $d_j$-dimensional. The maximal ideal of $R_{P_j}$ is compatible and therefore the socle of $H^{d_j}_{P_j}(R_{P_j})$ is a $1$-dimensional Frobenius stable submodule by Lemma~\ref{Lemma Frobenius stable submodules corresponds to compatible ideals}. If $\eta_j\in H^{d_j}_{P_j}(R_{P_j})$ generates the socle then there exists a $u_j\in R_{P_j}$ such that $\eta_j^p=u_j\eta_j$. Because $R$ is $F$-pure, in particular $F$-injective, the element $u_j$ is a unit of $R_{P_j}$. By clearing denominators and replacing $\eta_j$ by a suitable multiple of itself by a unit of $R_{P_j}$, we may assume that $u_j\in R$.

We claim that we may alter the element $u_j$ so that $\eta_j^p=u_j\eta_j$ and $u_j$ is a unit of $R_{Q_i}$ for each $1\leq i \leq m$. Suppose $\{\fq_1,\fq_2,\ldots,\fq_\ell\}$ is the collection of maximal elements of $\mathcal{Q}$ with respect to inclusion and has been written so that $u_j$ avoids $\fq_1\cup\cdots \cup \fq_i$ but $u_j$ is an element of $\fq_{i+1}\cap \cdots \cap \fq_m$. The prime ideals $\fq_1,\ldots, \fq_\ell$ are mutually incomparable and no $P_i$ is contained in some $\fq_n$. So we can choose an element $a\in P_1\cap \cdots \cap P_t\cap \fq_1\cap \cdots \cap \fq_i$ which avoids $\fq_{i+1}\cup \cdots \cup \fq_m$. Observe that $a\eta_j=0$ and therefore we may replace $u_j$ by $u_j+a$ and still have that $\eta_j^p=u_j\eta_j$. The element $u_j$ now avoids every element of $\mathcal{Q}$, i.e. $u_j$ is a unit of the localizations $R_{Q_i}$ for each $1\leq i\leq m$.


Prime avoidance allows us to choose parameters $x_1,\ldots,x_n$ of $R$ with the following properties:
\begin{enumerate}
\item if $R_{P_j}$ is $d_j$-dimensional then $x_1,\ldots, x_{d_j}$ is a system of parameters of $R_{P_j}$;
\item each $x_i$ avoids every element of $\mathcal{Q}$.
\end{enumerate}

Suppose that $R_{P_j}$ is $d_j$-dimensional and $\check{C}^\bullet(x_1,\ldots,x_{d_j};R_{P_j})$ is the \v{C}ech complex on $x_1,\ldots,x_{d_j}$. We realize the local cohomology module $H^{d_j}_{P_j}(R_{P_j})$ as the \v{C}ech cohomology module $H^{d_j}(\check{C}^\bullet(x_1,\ldots,x_{d_j};R_{P_j}))$. We choose an element
\[
\alpha_j\in \check{C}^d(x_1,\ldots,x_{d_j};R_{P_j})=(R_{P_j})_{x_1\cdots x_{d_j}}
\]
which is a representative of $\eta_j$.

Now, it is possible that some $d_j = 0$ (if the corresponding $P_j$ is minimal) and so suppose in fact that $d_1, \dots, d_a = 0$ with all other $d_i > 0$.  Fix $\fp_1, \dots \fp_e$ the minimal primes of $R$ not among the $P_1,  \dots, P_a$.  Let $R_1 = R/(\fp_1 \cap \dots \cap \fp_e)$.  By hypothesis, no $Q_i$ contains any $P_1, \dots, P_a$, and in particular, $R_{Q_i} \cong (R_1)_{Q_i}$.  Note if $I$ is not contained in any minimal prime then $R = R_1$.  It is also worth remarking that the image of the evaluation map 
\begin{equation}
    \label{eq.ImageOfR1ALreadyContainedInMinimalPrimes}
    \mathrm{Image}\big(\Hom_R(R_1, R) \to R\big) \subseteq   P_1 \cap \dots \cap P_a
\end{equation}
since $(R_1)_{P_j} = 0$ for $j = 1, \dots, a$.

Next, for all $d_j > 0$ with $j > a$, we set $g_j(T)=T^p-u_jT$, a monic polynomial over $R$. Then $g_j(\eta_j)=0$ and so there exists $\beta_j\in \check{C}^{d_j-1}(x_1,\ldots,x_{d_j};R_{P_j})$ so that $g_j(\alpha_j)=\partial^{d_j-1}(\beta_j)$. 
Suppose that
\[
\beta_j=\left(\frac{r_{i,j}}{x_1^{j_1}\cdots \widehat{x}_i\cdots x_{d_j}^{j_{d_j}}}\right)_{i=1}^{d_j}.
\]
Let $\{T_{i,j}\}_{_{1\leq i\leq d_j}^{a+1\leq j\leq t}}$ be variables and consider the single variable polynomials
\[
f_{i,j}:= g_j\left(\frac{T_{i,j}}{x_1^{j_1}\cdots \widehat{x_i}\cdots x_{d_j}^{j_{d_j}}}\right)-\frac{r_{i,j}}{x_1^{j_1}\cdots \widehat{x_i}\cdots x_{d_j}^{j_{d_j}}}.
\]
Multiplying $f_{i,j}$ by $(x_1^{j_1}\cdots \widehat{x_i}\cdots x_{d_j}^{j_{d_j}})^p$ produces a monic polynomial $\tilde{f}_{i,j}$ in $R_1[T_{i,j}]$ so that
\[
\frac{d\tilde{f}_{i,j}}{dT_{i,j}}=-u_j(x_1^{j_1}\cdots \widehat{x_i}\cdots x_{d_j}^{j_{d_j}})^{p-1}.
\]
Observe that each of the derivatives $\frac{d\tilde{f}_{i,j}}{dT_{i,j}}$ are units in the localized rings $R_{Q} = (R_1)_Q$ for each $Q\in \mathcal{Q}$. Therefore the $R$-algebra $R'=R_1[T_{i,j}]_{_{1\leq i\leq d_j}^{a+1\leq j\leq t}}/(\tilde{f}_{i,j})_{_{1\leq i\leq d_j}^{a+1\leq j\leq t}}$ is \'{e}tale over $R$ when localized at each element of $\mathcal{Q}$. Denote by $t_{i,j}$ the images of $T_{i,j}$ in $R'$.

In what follows, elements of the total ring of fractions of $R$ map to the localization of $R'$ at the set of non-zero divisors of $R$, and we identify them with their images.
Consider the following elements of $\check{C}^{d_{j}-1}(x_1,\ldots,x_{d_j}; R'_{P_j})$ and $\check{C}^{d_{j}}(x_1,\ldots,x_{d_j}; R'_{P_j})$ respectively (for $j >a$):
\begin{enumerate}
\item $\overline{\beta}_j=\left(\frac{t_{i,j}}{x_1^{j_1}\cdots \widehat{x_i}\cdots x_{d_j}^{j_{d_j}}}\right)_{i=1}^{d_j}$;
\item $\overline{\alpha}_j=\alpha_j-\partial^{d_j-1}(\overline{\beta}_j)$.
\end{enumerate}
Then $g_j(\overline{\beta}_j)=\beta_j$. Raising to $p^{th}$ powers is additive and therefore $g_j$ is additive. It follows that
\begin{equation}
\label{0 element}
g_j(\overline{\alpha}_j)=g_j(\alpha_j)-g_j(\partial^{d_j-1}(\overline{\beta}_j))=\partial^{d_j-1}(\beta_j)-\partial^{d_j-1}(g_j(\overline{\beta}_j))=0.
\end{equation}
Therefore $\overline{\alpha}_j$ is an element of the total ring of fractions of $R'$ satisfying the monic polynomial $g_j(T)$. 

The element $\overline{\alpha}_j$ belongs to $(R'_{P_j})_{x_1\cdots x_{d_j}}$. Therefore there exists an element $r_j\in R'$, $s_j\in R\setminus P_j$, and natural numbers $\ell_{1_j},...\ell_{d_j}$ such that
\[
\overline{\alpha}_j=\frac{r_j/s_j}{x_1^{\ell_{1_j}}\cdots x_{d_j}^{\ell_{d_j}}}.
\]
Let $\widetilde{\alpha}_j=s_j\overline{\alpha}_j$ and $\widetilde{g_j}(T)=T^p-s_j^{p-1}u_jT$. Multiplying (\ref{0 element}) by $s_j^p$ shows that  $\widetilde{g_j}(\widetilde{\alpha}_j)=0$. The parameters $x_1,x_2,\ldots,x_{d_j}$ avoid each $Q\in\mathcal{Q}$. Therefore, for each $Q\in\mathcal{Q}$ we have
\begin{equation}
\label{This element lives in a localized ring}
\widetilde{\alpha}_j\in R'_{x_1\cdots x_{d_j}}\subseteq R'_Q.
\end{equation}
Even further, $\widetilde{\alpha}_j$ is a representative of $s_j\eta\in H^d_{P_j}(R')$.
Let $S=R'[\widetilde{\alpha}_j]_{j=a+1}^t$. Then $R\to S$ is a finite map and is an extension if $I$ is not contain in a minimal prime.  If $I$ is contained in a minimal prime $\fp$ then as observed above $(R_1)_{\fp} = 0$ and so $S'_{\fp} = 0$ as well.  Regardless, $H^{d_j}_{P_j}(R)\to H^{d_j}_{P_j}(S)$ maps $s_j\eta$ to $0$ for each $a+1\leq j\leq t$. However, the element $s_j\in R_{P_j}$ is a unit, and therefore $\eta$ is mapped to the $0$-element of $H^{d_j}_{P_j}(S)$ for each $a+1\leq j\leq t$.


By Lemma~\ref{Lemma checking containment}, Lemma~\ref{Lemma Frobenius stable submodules corresponds to compatible ideals} and equation \eqref{eq.ImageOfR1ALreadyContainedInMinimalPrimes} above, the trace ideal $\tau_{S/R}$ is contained in $I$. By Lemma~\ref{Lemma image ideals are compatible} we know that $\tau_{S/R}$ is compatible. In particular, the trace ideal is a radical ideal and it remains to observe that $I$ is not contained in any element of $\mathcal{Q}$. If $Q\in \mathcal{Q}$ then $S_Q=(R'[\overline{\alpha_j}]_{j=a}^t)_Q=R'_Q$ by (\ref{This element lives in a localized ring}). The map $R\to R'$ is \'{e}tale at $Q$ and hence the trace ideal of $R\to S$ at $Q$ agrees with the unit ideal.

Finally, we need to explain why we may choose $S$ a domain if $R$ is a normal domain and $I \neq 0$.  Let $\fp$ be a minimal prime of $S$ contracting to $(0) \subseteq R$ and consider the extension $R \subseteq S \to S/\fp$.  Certainly we have that the trace ideal of $R \to S/\fp$ is contained in $I$, but we need to explain why it is not smaller.  However, for any $Q_i \in \mathcal{Q}$, $R_{Q_i} \to S_{Q_i}$ is \'etale.  Hence since $R_{Q_i}$ is normal, so is $S_{Q_i}$ and so $\widehat{S_{Q_i}}$ is a product of domains, each is \'etale over $\widehat{R_{Q_i}}$. 
Hence $R_{Q_i} \to S_{Q_i}/\fp_{Q_i}$ is also finite \'etale since it completes to an \'etale map.  It follows that the trace ideal of $R \subseteq S/\fp$ equals $I$.
\end{proof}

Our goal for the rest of the section is to generalize Theorem \ref{Prime ideals in a quasi-Gorenstein ring} to the $\bQ$-Gorenstein case. To accomplish this we consider properties of cyclic covers of $R$ associated to the canonical module.  Useful references on cyclic covers in this generality include \cite{CarvajalRojasFiniteTorsors}, \cite{KollarKovacsSingularitiesBook} and \cite[Appendix A]{MaPolstraFBook}. We begin with some lemmas very closely related to work of Speyer \cite{SpeyerFrobeniusSplitSubvarieties}, also {\it c.f.} \cite{SchwedeTuckerTestIdealFiniteMaps}.

We first prove the following results about cyclic covers of index not divisible by $p > 0$ which we assume are known to experts but for which we know no reference (in the case that $R$ is not necessarily normal).

\begin{lemma}
    \label{lem.TrVsTr}
    Suppose $R$ is a G1 and S2 reduced local ring of characteristic $p > 0$, $N$ is a Weil divisorial module of index $n$ not divisible by $p > 0$.  Let $T \cong R \oplus N \oplus N^{(2)} \oplus \dots \oplus N^{(n-1)}$ denote an unramified-in-codimension-1 cyclic cover (of index $n$) with respect to $N$.  Let $\Tr : T \to R$ denote the projection onto degree $0$ (the projection onto the $R$-summand) and let $\mathbb{T} : K(T) \to K(R)$ denote the trace map on the level of total rings of fractions.  Then $\mathbb{T} : K(T) \to K(R)$ restricts to a map $\mathbb{T} : T \to R$.  Furthermore, this map agrees with $\Tr$ up to multiplication by a unit of $T$ and hence 
    \[
        \mathbb{T} \in \Hom_R(T, R)
    \]
    generates $\Hom_R(T, R)$ as a $T$-module.  
\end{lemma}
\begin{proof}    
    Since $R \subseteq S$ is \'etale in codimension 1 it is locally free in codimension $1$.  Hence $\mathbb{T}$ induces a map $T \to S$ at least in codimension 1.  But now since $T$ and $S$ are both S2, we have that $\mathbb{T} \in \Hom_R(T, R)$.  

    We know from \cite[\href{https://stacks.math.columbia.edu/tag/0BT8}{Tag 0BT8}]{stacks-project} that $\mathbb{T}$ is the trace element in codimension 1 and so by \cite[\href{https://stacks.math.columbia.edu/tag/0BW9}{Tag 0BW9}]{stacks-project} $\mathbb{T} \in \Hom_R(T, R)$ generates $\Hom_R(T, R)$ as a $T$-module in codimension-1 wherever $R \subseteq T$ is \'etale.  Since $\Hom_R(T, R)$ is an S2 $T$-module of rank $1$,  we see that $\mathbb{T} \in \Hom_R(T, R)$ is a $T$-module generator.  But $\Tr \in \Hom_R(T, R)$ also generates as a $T$-module by \cite[Lemma A.4]{MaPolstraFBook} and so $\Tr$ and $\mathbb{T}$ agree up to multiplication by a unit of $T$.
\end{proof}

\begin{lemma}
    \label{lem.TraceSendsRadicalToRadical}
    Suppose $R \subseteq T$ is as in Lemma \ref{lem.TrVsTr} and additionally assume that $R$ is excellent (for instance if $R$ is $F$-finite). Suppose $I$ is a radical ideal of $R$.  Then $\mathbb{T}(\sqrt{IT}) \subseteq I$.
\end{lemma}
\begin{proof}
    In the case that $R$ is normal, this is simply \cite[Lemma 9]{SpeyerFrobeniusSplitSubvarieties}.  We reduce to that case as follows.  Indeed, let $R^{\mathrm{N}} \supseteq R$ and $T^{\mathrm{N}} \supseteq T$ denote the normalizations of $R$ and $T$.  Then $\mathbb{T}(\sqrt{I T^{\mathrm{N}}}) \subseteq \sqrt{I R^{\mathrm{N}}}$ as already observed.  But $\sqrt{I R^{\mathrm{N}}} \cap R = I$ and $\sqrt{I T^{\mathrm{N}}} \cap T = \sqrt{IT}$ and the result follows.
\end{proof}

\begin{lemma}
    \label{lem.CompatibleOnFiniteCoverEqualsCompatible}
    Suppose that $R$ is $F$-finite and $F$-pure ring of prime characteristic $p>0$ and $N$ a torsion Weil divisorial module whose index $n$ is not divisible by $p > 0$.  Let $R\to T$ be a cyclic cover of $R$ with respect to $N$. Fix a surjective map $\phi :F^e_* R \to R$ and extend it to a map $\phi_T : F^e_* T \to T$.  Then $I \subseteq R$ is compatible with $\phi$ if and only if $\sqrt{IT} \subseteq T$ is compatible with $\phi_T$.  
\end{lemma}
\begin{proof}
    Without loss of generality, we may assume that $R$ is local since compatibility may be checked locally.
    Let $\Tr : T \to R$ denote the the projection onto the degree 0 piece.  As an $T$-module, $\Hom_R(T,R)\cong T$ with $T$-module generator given by $\Tr$.
    Notice that there exists a map $\phi_T$ fitting into a diagram
    \[
        \xymatrix{
            F^e_* T \ar[r]^{\phi_T'} \ar[d]_{\Tr} & T \ar[d]^{\Tr} \\
            F^e_* R \ar[r]^{\phi} & R 
        }
    \]     
    by \cite[Proposition 5.7]{MaPolstraFBook} (see \cite{SchwedeTuckerTestIdealFiniteMaps} or \cite{SpeyerFrobeniusSplitSubvarieties} for the case when $R$ is normal).  Since $\Tr$ agrees with $\mathbb{T}$ up to multiplication by a unit $\mathbb{T} = \Tr \circ u$, by Lemma \ref{lem.TraceSendsRadicalToRadical}, we obtain the commutative diagram:
    \[
        \xymatrix{
            F^e_* T \ar[r]^{\phi_T} \ar[d]_{\mathbb{T}} & T \ar[d]^{\mathbb{T}} \\
            F^e_* R \ar[r]^{\phi} & R 
        }
    \]  
    setting $\phi_T = \phi_T' \circ F^e_* u^{1-p^e}$.  By \cite{SchwedeTuckerTestIdealFiniteMaps} or \cite{SpeyerFrobeniusSplitSubvarieties} we see that $\phi_T|_{F^e_* R} = \phi$. 
    %
    Now, if $\sqrt{IT}\subseteq T$ is $\phi_T$-compatible, then it is clear that $I = \sqrt{IT} \cap R \subseteq R$ is $\phi$-compatible.  So we assume that $I$ is $\phi$-compatible and aim to establish that $\sqrt{IT}$ is $\phi_T$-compatible.

 The map $\phi$ is surjective and therefore $I$ is a radical ideal of $R$. To study compatibility of the ideal $\sqrt{IT}$ we replace $R$ by a localization and completion at a minimal prime $Q$ of $I$, so that $(R, \fm)$ is local and $I = \fm$. In this case, via base change, $T = \prod T_i$ becomes a finite product of complete local rings $T_i$. The trace map $\mathbb{T} : T \to R$ onto $R$ simply sums over the trace maps of the individual terms in the product $T_i$.  By restricting $\phi_T$ to each $T_i$, it then suffices to handle each $T_i$ separately and so we may assume that $T$ is itself local with maximal ideal $\fn = \sqrt{\fm T}$.

As we argued at the start of the proof, there exists a commutative diagram
    \[
        \xymatrix{
            F^e_* T \ar[d]_{F^e_*\mathbb{T}}\ar[r]^{\phi_T} & T \ar[d]^{\mathbb{T}}\\
            F^e_* R \ar[r]^{\phi} & R
        }
    \]
Suppose $\fm$ is compatible with $\phi$.  If $\phi_T(F^e_* \fn) \not\subseteq \fn$ then $\phi_T(F^e_* \fn) = T$ which implies that
    \[
        R = \mathbb{T}(T) = \mathbb{T}\phi_T(F^e_* \fn) = \phi(F^e_* \mathbb{T}(\fn)) = \phi(F^e_* \fm) \subseteq \fm, 
    \]
    using Lemma \ref{lem.TraceSendsRadicalToRadical}, a contradiction.
\end{proof}

We now describe how to handle cases where the cyclic cover has a $p$th power index. 

\begin{lemma}
\label{lemma compatible ideals cyclic covers p^e}
    Suppose that $R$ is a G1+S2 $F$-finite $F$-pure local ring of prime characteristic $p > 0$ and suppose that $N$ is a Weil divisorial module of index $p^{e_0}$.  Let $R \to T$ denote a cyclic cover of $R$ with respect to $N$.  Suppose $I \subseteq R$ is a uniformly compatible ideal of $R$.  Then $I_T = \sqrt{IT}$ is a uniformly compatible ideal of $T$.
\end{lemma}
\begin{proof}    
    Compatible ideals in $F$-pure rings are radical and a radical ideal is compatible if and only if each of the components is compatible.  We therefore may assume that $I=Q$ is a uniformly compatible prime ideal of $R$.  We set $Q_T = \sqrt{QT}$ and notice that $Q_T$ is prime since $R \subseteq T$ is purely inseparable. 
    By \cite[Proposition 4.20]{CarvajalRojasFiniteTorsors} or \cite[Appendix A]{MaPolstraFBook} we know that $T$ is $F$-pure.  By construction we have $R \subseteq T \subseteq R^{1/p^{e_0}}$ and the second inclusion map $T \to R^{1/p^{e_0}}$ splits since $T \to T^{1/p^{e_0}}$ splits.  Let $\Tr : T \to R$ denote the projection onto the first coordinate.

    Consider an arbitrary map $\phi \in \Hom_T(T^{1/p^e}, T)$.  We need to show that $\phi(Q_T^{1/p^e}) \subseteq Q_T$.  This may be checked after localizing at $Q$ and so we assume that $R$ and $T$ are local with maximal ideals $Q$ and $\sqrt{QT}$ respectively. 
    Continue to let $\psi: R^{1/p^{e_0}}\to T$ be a $T$-linear splitting of $T\subseteq R^{1/p^{e_0}}$. There exists a commutative diagram:
    \[
        \xymatrix{
           T^{1/p^e}\ar[r]\ar[d]^{\phi} & R^{1/p^{e+e_0}}\ar[r]^{\psi^{1/p^e}} & T^{1/p^e}\ar[dll]^{\phi'} \\
           T 
        }
    \]
    In particular, $\phi:T^{1/p^e}\to T$ can be lifted to a map $\varphi: R^{1/p^{e+e_0}}\to T$ (the restriction of $\phi'$). Therefore, if it were not the case that $\phi\big(\sqrt{QT^{1/p^e}}\big)\not\subseteq \sqrt{QT}$, then $\phi\big(\sqrt{QT^{1/p^e}}\big) = T$ and, because $\Tr : T \to R$ is surjective, the composition of the maps
    \[
    R^{1/p^{e+e_0}}\xrightarrow{\varphi} T \xrightarrow{\Tr} R
    \]
    has the property that $\Tr(\varphi(Q^{1/p^{e_0+e}})) = \Tr\circ \varphi\big(\sqrt{QR^{1/p^{e+e_0}}}\big)\not\subseteq Q$, a contradiction as $Q\subseteq R$ is uniformly compatible.
\end{proof}

\begin{remark}
    Unlike the \'etale in codimension $1$ scenario, see Lemma~\ref{lem.CompatibleOnFiniteCoverEqualsCompatible}, we do not know if the converse to Lemma \ref{lemma compatible ideals cyclic covers p^e} holds.  That is if $\sqrt{IT}$ is uniformly compatible, is $I$ uniformly compatible?
\end{remark}

Combining the previous two lemmas, we obtain the following.
\begin{lemma}
    \label{lem.CompatibleOnFiniteCoverEqualsCompatibleGeneralCase}
    Suppose that $R$ is a G1+S2 $F$-pure $F$-finite local ring and $I \subseteq R$ is a uniformly compatible ideal.  Viewing $\omega_R$ as a Weil divisorial module of index $n$ we let $T$ is be an associated cyclic cover (a canonical cover).  Then $\sqrt{IT}$ is uniformly compatible on $T$.
\end{lemma}
\begin{proof}
    Suppose that $n = p^{e_0} m$ where $p$ does not divide $m$.  Let $T'$ be the $m$th Veronese subring of $T$, in particular it is a cylic cover associated to $\omega_R^{(m)} \cong R(mK_R)$.  We know $I_{T'} := \sqrt{IT'}$ is uniformly compatible by Lemma \ref{lemma compatible ideals cyclic covers p^e}.  Furthermore, $T'$ is $\bQ$-Gorenstein with index not divisible by $p > 0$.  Let $\Phi \in \Hom_{T'}({T'}^{1/p^{e_0}}, T')$ be a ${T'}^{1/p^{e_0}}$-module generator.  
    Now, $T' \subseteq T$ is a cyclic cover of index $m$ and so 
    \[
        \sqrt{IT} = \sqrt{I_{T'} T}
    \]
    is $\Phi_T$-compatible by Lemma \ref{lem.CompatibleOnFiniteCoverEqualsCompatible} where $\Phi_T$ is the unique extension of $\Phi$ to $T$ (that is, $\Phi_T$ is simply the reflexification of $\Phi \otimes_{T'} T$).  But then since $T' \subseteq T$ is \'etale in codimension $1$ we see that the map
    \[
        \Hom_{T'}(T'^{1/p^{e_0}}, T') \otimes_{T'} T \to \Hom_T(T^{1/p^{e_0}}, T)    
    \]
    is an isomorphism up to reflexification.  
    Hence we see that $\Phi_T$ generates $\Hom_T(T^{1/p^{e_0}}, T)$ as a $T^{1/p^{e_0}}$-module.  Finally, by \cite[Proposition 4.1]{SchwedeFAdjunction}\footnote{In that reference, it is assumed that the ring is normal, but normality is not used in the proof.} we see that $\sqrt{IT}$ is uniformly compatible.  
\end{proof}

    The next lemma crucially uses that the ideal $I\subseteq R$ of the lemma is assumed to be uniformly compatible on $R$.
    \begin{lemma}
    \label{Trace of QT}
        Suppose $R$ is a G1+S2 $F$-pure $F$-finite local ring and $M$ is a Weil-divisorial module of index $n$.  Let $R \subseteq T$ be an associated cyclic cover (of degree $n$).
        Suppose $I \subseteq R$ is a uniformly compatible ideal.  Then $\Tr(\sqrt{IT}) = I$ where $\Tr$ is the projection onto the degree 0 piece.
    \end{lemma}
    \begin{proof}
        Without loss of generality $I = Q$ is a prime ideal.  
        We first handle the case that $n = p^{e_0}$.  
        The extension $R\to T$ splits and so clearly $Q\subseteq \sqrt{QT} =: Q_T$. To prove the reverse containment, using that $T$ is $F$-split and that $R \subseteq T$ is purely inseparable, we may choose $\psi \in \Hom_{T}(R^{1/p^{e_0}}, T)$ so that the composition
        \[
            T \hookrightarrow R^{1/p^{e_0}} \xrightarrow{\psi} T 
        \]
        is the identity.  Hence $\Tr \circ \psi \in \Hom_R(R^{1/p^{e_0}}, R)$ and since $Q$ is uniformly compatible, we have that $\Tr(\psi(Q^{1/p^{e_0}})) \subseteq Q$.  But $Q_T R^{1/p^{e_0}} \subseteq Q^{1/p^{e_0}}$ and the claim follows.
        
        For the general case, we can break up our extension into $R \subseteq T' \subseteq T$ where $R \subseteq T'$ is a cyclic cover of order $p^{e_0}$ and $T' \subseteq T$ is a cyclic cover whose index is not divisible by $p$.  Let $\Tr_{T'/R} : T' \to R$ (and likewise with $\Tr_{T/R}$ and $\Tr_{T/T'}$) be the projection onto the degree zero piece.  We notice that $\Tr_{T/R} = \Tr_{T'/R} \circ \Tr_{T/T'}$.  Hence,  in view of our first paragraph, we need only prove that $\Tr_{T/T'}(\sqrt{JT}) \subseteq J$ for any radical ideal $J \subseteq T'$ (we will take $J = \sqrt{QT'} = Q_{T'}$).

        Since $\Tr$ and the total-ring-of-fractions-trace $\mathbb{T}$ agree up to multiplication by a unit, it suffices to prove this for $\mathbb{T}$.  Now, $\mathbb{T}$ induces a map between the normalizations $\mathbb{T} : T^{\mathrm{N}} \to T'^{\mathrm{N}}$ and that map sends $\sqrt{JT^{\mathrm{N}}}$ into $\sqrt{JT'^{\mathrm{N}}}$ by \cite[Lemma 9]{SpeyerFrobeniusSplitSubvarieties}.  The result follows by intersecting these ideals with $T$ and $T'$ respectively.  In particular, we have shown that $\Tr_{T/T'}(\sqrt{JT}) \subseteq J$ for any radical ideal $J \subseteq T'$.  The lemma follows.
    \end{proof}

\begin{corollary}
    \label{cor.GeneralCase}
    Suppose that $R$ is an $F$-finite $\bQ$-Gorenstein $F$-pure local ring.  Then for every compatible ideal $I \subseteq R$, there exists a finite ring map $R \to S$ (which we may take to be an extension if $I$ is not contained in any minimal prime) such that $I$ is the trace ideal,
    \[
        I = \tau_{S/R}.
    \]
    Finally, if $R$ is a normal domain and the $\bQ$-Gorenstein index of $R$ is not divisible by $p > 0$, then we may take $S$ to be a domain.
\end{corollary}
\begin{proof}
    Let $R \subseteq T$ be a cyclic cover associated to $K_R$.   For every finite composition $R \subseteq T \to S$, we have $\Hom_R(S, R) \to \Hom_R(T, R) \to R$ where the second map may be identified with $\Tr : T \to R$.  
    But now $T$ is quasi-Gorenstein, $\sqrt{IT}$ is a uniformly compatible ideal of $T$ by Lemma \ref{lem.CompatibleOnFiniteCoverEqualsCompatibleGeneralCase}, and so there exists a finite ring map $T\to S$ so that $\tau_{S/T}=\sqrt{IT}$ by Theorem~\ref{Prime ideals in a quasi-Gorenstein ring} (and $T \to S$ may be taken to be an extension if $I$ is not contained in any minimal prime since then $\sqrt{IT}$ is also not contained in a minimal prime). We claim that $\tau_{S/R}=I$. To see this, recall again that $\Hom_R(T,R)\cong T$ is principally generated as a $T$-module by $\Tr: T\to R$. Hence, by $\Hom$-tensor adjunction,
    \[
    \Hom_S(S,T)\cong \Hom_S(S,\Hom_R(T,R))\cong \Hom_R(S,R).
    \] 
    Therefore $\tau_{S/R}=\Tr(\tau_{S/T}) = \Tr(\sqrt{IT})$ as every $R$-linear map $S\to R$ factors through a $T$-linear map $S\to T$. But $\Tr(\sqrt{IT})=I$ by Lemma~\ref{Trace of QT}.

    For the final statement, we observe we have $R \subseteq T$ a cyclic cover of index prime to $p$.  Since $R$ is normal, $T$ is a normal domain by \cite[Corollary 1.9]{TomariWatanabeNormalZrGradedRings}.  Hence we may assume $S$ is a domain by Theorem~\ref{Prime ideals in a quasi-Gorenstein ring}.
\end{proof}

\section{Further questions and examples}

There are at least two ways which one could try generalize these results, we could try to remove the $\bQ$-Gorenstein hypothesis, or we could try to weaken the $F$-purity hypothesis.

\begin{question}[Removing the $\bQ$-Gorenstein hypothesis]
    Suppose that $R$ is an $F$-finite $F$-pure local ring.  Is every uniformly compatible ideal always the trace ideal of some finite ring map?
\end{question}

Our proof doesn't seem to generalize to this setting.  Indeed, as mentioned in the introduction, a positive answer to this question would imply that splinters (and hence weakly $F$-regular rings) are strongly $F$-regular.

Alternately, we might try to weaken the $F$-purity hypothesis.  For simplicity, let us work in the quasi-Gorenstein setting which implies that there there exists $\Phi \in \Hom_R(F_* R, R)$ generating the $\Hom$-set.  In this case, the unifomly compatible ideals are exactly the $\Phi$-compatible ideals and there can be infinitely many of them.  However, there is a distinguished \emph{finite} set of $\Phi$-compatible ideals, the \emph{$\Phi$-fixed ideals}.  In this setting, an ideal is ($\Phi$-)fixed if 
\[
    \Phi(F_* I) = I,
\]
instead of merely $\subseteq$.  See \cite{BlickleBockleCartierModulesFiniteness} for the fact that there are only finitely many such ideals and the test ideal is the smallest such ideal not contained in any minimal prime.  It is also easy to see that for a surjective $\Phi$ (the case when $R$ is $F$-pure) the compatible ideals are always fixed.  

\begin{question}[Weakening the $F$-purity hypothesis]
    Suppose $(R, \fm)$ is an $F$-finite quasi-Gorenstein local ring.   Is every fixed ideal $I$ always the trace ideal of some finite ring map?
\end{question}

If one studies this question in the local case where $\sqrt{I} = \fm$ and $k$ is perfect, then there is an associated finitely generated Frobenius fixed $k$-subvector space of $H^d_{\fm}(R)$ corresponding to $I$.  However, we were not able to gain enough control over the relations induced by the equational lemma to mimic our approach in Theorem \ref{Prime ideals in a quasi-Gorenstein ring}.

There is an interesting and challenging related example worth considering coming out of \cite{BlickleThesis}.  Consider $k[x_1, x_2, x_3]/\langle x_1^4 + x_2^4 + x_3^4 \rangle$ with maximal ideal $\fm = (x_1,x_2,x_3)$.  In this case, the test ideal is $\fm^2$ and $\fm$ is also fixed.  But so is $\fm^2 + (x_i + ax_j)$ for any $a \in \bF_p$.  However, if one takes $\Phi^e \in \Hom_R(F^e_* R, R)$ generating the $\Hom$-set, then $\fm^2 + (x_i + ax_j)$ is $\Phi^e$-fixed for any $a \in \bF_{p^e}$.  It is not clear whether we should expect all of these ideals to be trace ideals, or only some of them.  Note these intermediate ideals do not show up in the corresponding characteristic zero picture as studied in \cite{HsiaoSchwedeZhang}.


\bibliographystyle{skalpha}
\bibliography{MainBib}

\end{document}